\documentclass[runningheads]{llncs}

\usepackage[T1]{fontenc}
\usepackage{amsmath,amssymb}
\usepackage{hyperref}

\begin{document}

\title{Non-Definability of Reachability in B\"uchi Arithmetic for a Family of Generalized Collatz Maps}

\author{Madhav Dhiman\inst{1} \and Rohan Pandey\inst{2}}
\authorrunning{M. Dhiman and R. Pandey}

\institute{
Independent Researcher \\
\email{mmdhiman09@gmail.com} \and
University of Washington \\
\email{rpande@uw.edu}
}

\maketitle

\begin{abstract}
Let $q \ge 3$ and $d \ge 1$ be odd integers with $q+d$ a power of $2$.
We study the generalized Collatz map $T_{q,d}$, a one-dimensional
piecewise-affine map on the positive integers, and its unparameterized
reachability relation $R(x,z)$, which holds when $z$ is an iterate of
$x$ under $T_{q,d}$. We prove that for every such pair $(q,d)$ the
relation $R$ is not first-order definable in B\"uchi arithmetic
$\langle \mathbb{N}, +, V_q \rangle$. Equivalently, no finite automaton
recognizes the base-$q$ encoding of $R$. Assuming definability of $R$,
we construct a first-order formula that defines the set of powers of
$2$. Cobham's theorem then rules out this set. The family includes the
classical map $T_{3,1}$. The family is infinite but restricted, isolated by the condition that $q+d$ is a power of $2$ . Unlike the undecidability results
of Conway, Kurtz, and Simon, the construction does not embed universal
computation and does not depend on the Collatz conjecture.
\end{abstract}

\section{Introduction}

The generalized Collatz map $T_{q,d}$ with odd parameters $q \ge 3$,
$d \ge 1$ is the one-dimensional piecewise-affine map (PAM) on
$\mathbb{Z}_{>0}$ defined by
\[
T_{q,d}(n) = \begin{cases} n/2 & \text{if } n \text{ is even,} \\
(qn+d)/2 & \text{if } n \text{ is odd.} \end{cases}
\]
The classical Collatz map is $T_{3,1}$. Reachability for
one-dimensional integer PAMs is a central topic in the field of
reachability problems~\cite{BenAmram,Blondel}. The Collatz family is a
well-known difficult case.

Let $R \subseteq \mathbb{Z}_{>0}^2$ be the unparameterized reachability
relation: $(x,z) \in R$ if $z = T_{q,d}^k(x)$ for some $k \ge 0$. We
study base-$q$ B\"uchi arithmetic $\mathrm{BA}_q = \langle \mathbb{N},
+, V_q \rangle$, where $V_q(n)$ is the largest power of $q$ dividing
$n$. A set is first-order definable in $\mathrm{BA}_q$ if and only if it
is recognized by a finite automaton reading base-$q$
representations~\cite{Buchi,Bruyere}.

\paragraph{Main result.}
We prove the following. Let $q \ge 3$ and $d \ge 1$ be odd integers
with $q+d$ a power of $2$. Then $R$ is not first-order definable in
$\mathrm{BA}_q$. The family is infinite. For each $s \ge 2$ it contains
every pair $(q,d)$ of odd integers with $q+d = 2^s$, including the
classical map $T_{3,1}$ ($q+d = 4$). What is new is the non-$q$-recognizability of the unparameterized reachability relation, established for an explicit infinite family via a single first-order extraction formula, without using universal computation and independent of the Collatz conjecture.

\paragraph{Why the proof is not immediate.}
One might assume that combining multiplication by $q$ and division by
$2$ trivially prevents base-$q$ definability. That argument applies
only when the iteration count is given. The relation $R$ is
unparameterized, so a finite automaton for $R$ evaluates pairs $(x,z)$
without counting iterations. Our proof constructs a first-order formula
in $\mathrm{BA}_q$ that defines the set of powers of $2$ from $R$.
Cobham's theorem then shows this set is not $q$-recognizable. The
formula construction is the core of the proof.

\paragraph{Why the family is restricted.}
The condition $q+d = 2^s$ forces the cycle reachable from $1$ to be the pure halving cycle. This restriction avoids the definability collapse seen in unrestricted parameters (detailed in Section~\ref{sec:scope}), enabling our formula to precisely isolate the descent of powers of $2$. 

\paragraph{Comparison with prior undecidability results.}
Conway~\cite{Conway} and Kurtz and Simon~\cite{KurtzSimon} studied the
family of generalized Collatz functions. They embedded universal
computation to prove that the halting problem for this family is
undecidable. Their results apply to the entire family of functions. We
study a different question. We fix a single map and ask whether its
reachability relation is recognizable by a finite
automaton~\cite{Bruyere}. Our proof uses a reachable floor bound to extract
the strictly decreasing trajectory of powers of $2$. It does not use
universal computation. It does not depend on the Collatz conjecture or
on any undecidability result.

\paragraph{Organization.}
Section~\ref{sec:related} reviews related work.
Section~\ref{sec:prelim} recalls $\mathrm{BA}_q$ and Cobham's theorem.
Section~\ref{sec:scope} sets out the family and the obstruction it
avoids. Section~\ref{sec:trajectory} develops the trajectory structure.
Section~\ref{sec:extraction} gives the logical extraction.
Section~\ref{sec:main} proves the main theorem.
Section~\ref{sec:discussion} places the result in context.

\section{Related Work}
\label{sec:related}

Reachability for piecewise-affine maps over $\mathbb{Z}$ is a core
reachability topic. Blondel and Tsitsiklis~\cite{Blondel} established
undecidability for broad classes of piecewise-affine systems.
Ben-Amram~\cite{BenAmram} proved that mortality of iterated
piecewise-affine functions over $\mathbb{Z}$ is undecidable.
The one-dimensional PAM family, of which
$T_{q,d}$ is a member, lies at the boundary of decidability for
reachability questions.

Conway~\cite{Conway} proved undecidability of the multiplicative
generalized Collatz problem by extracting the first power of $2$
reached. That idea prefigures our extraction. Kurtz and
Simon~\cite{KurtzSimon} placed the problem at $\Pi^0_2$-completeness.
Lagarias~\cite{Lagarias} surveys the classical $3x+1$ problem and its
generalizations. Terras~\cite{Terras} established density results for
parity vectors. Matthews and Watts~\cite{MatthewsWatts} extended the
algebraic structure to arbitrary bases. Kari~\cite{Kari} proved
undecidability of nilpotency for one-dimensional cellular automata, a
model example of a simple rule with hard dynamics; our separation sits
in that tradition but stays inside automaton recognizability rather than
undecidability.

B\"uchi~\cite{Buchi} established the correspondence between first-order
definability in $\mathrm{BA}_b$ and $b$-recognizability. Bruy\`ere,
Hansel, Michaux, and Villemaire~\cite{Bruyere} is the canonical
reference for the $p$-recognizable-sets framework that underlies our
argument. Cobham~\cite{Cobham} proved the base-dependence theorem we
use. Allouche and Shallit~\cite{AlloucheShallit} provide a treatment of
automatic sequences. Presburger~\cite{Presburger} established
decidability of additive arithmetic. Ginsburg and
Spanier~\cite{Ginsburg} characterized Presburger-definable sets as
semilinear sets. Rigo~\cite{Rigo} provides a modern
treatment.

\section{Preliminaries}
\label{sec:prelim}

\subsection{B\"uchi Arithmetic and $b$-Recognizability}

Base-$q$ B\"uchi arithmetic $\mathrm{BA}_q$ has domain $\mathbb{N}$,
with addition and the function $V_q(n)$, defined as the largest power of $q$
dividing $n$. By B\"uchi~\cite{Buchi} and Bruy\`ere et al.~\cite[Section~2]{Bruyere},
a set $S \subseteq \mathbb{N}^k$ is first-order definable in
$\mathrm{BA}_q$ if and only if it is recognized by a finite automaton
reading base-$q$ representations. This equivalence holds for every integer base $q \ge 2$, including odd and composite bases. While $R$ is defined on the positive integers, definability of $R$ refers to its definability as a relation over $\mathbb{N}$ with $0$ excluded by a positivity guard ($x > 0 \iff \exists y\,(x = y+1)$). Let $P = \{2^k \mid k \in \mathbb{N}\}$. We have $P \subseteq \mathbb{N}$ and $0 \notin P$.

The order is expressible: $x \ge y \iff \exists w\,(x = y+w)$.
Parity is expressible: $\mathrm{Even}(n) \iff \exists y\,(n = y+y)$.

We work in base $q$ rather than base $2$ for a specific reason. In base
$q$ the operation $x \mapsto qx$ is recognizable, so multiplication by
$q$ in the odd branch of $T_{q,d}$ is free for the automaton. The only
arithmetic that mixes bases is the division by $2$. Base $q$ therefore
isolates the obstruction we exploit and does not hide an easy
base-mixing objection. Base $2$ would make division by $2$ free but
multiplication by $q$ hard, so neither base trivializes the question.

\subsection{Cobham's Theorem}

Two integers $a,b \ge 2$ are multiplicatively independent if $a^m = b^n$
has no solution with $m,n \ge 1$. For $b \ge 2$, the set of powers of
$b$ is $b$-recognizable but not eventually periodic.

\begin{theorem}[Cobham~\cite{Cobham}]\label{thm:cobham}
Let $a,b \ge 2$ be multiplicatively independent integers. If a set $S
\subseteq \mathbb{N}$ is both $a$-recognizable and $b$-recognizable,
then $S$ is eventually periodic.
\end{theorem}

\begin{lemma}\label{lem:undefinable_subset}
Let $q \ge 3$ be an odd integer. The set of powers of $2$, $P = \{2^k
\mid k \in \mathbb{N}\}$, is not first-order definable in
$\mathrm{BA}_q$.
\end{lemma}
\begin{proof}
The base-$2$ representations of the elements of $P$ form the language
$1\,0^{*}$, so $P$ is $2$-recognizable. The set $P$ is not eventually
periodic, since the gaps $2^{k+1}-2^k$ grow without bound. Assume for
contradiction that $P$ is first-order definable in $\mathrm{BA}_q$. By
B\"uchi~\cite{Buchi} and Bruy\`ere et al.~\cite{Bruyere}, $P$ is then
$q$-recognizable. Since $q \ge 3$ is odd, the equation $2^m = q^n$ has
no solution with $m,n \ge 1$ by unique factorization~\cite{HardyWright},
so $2$ and $q$ are multiplicatively independent. By
Theorem~\ref{thm:cobham}, $P$ is eventually periodic, a contradiction.
Therefore $P$ is not first-order definable in $\mathrm{BA}_q$.
\end{proof}

\section{The Family and Its Scope}
\label{sec:scope}

Throughout the rest of the paper we fix odd integers $q \ge 3$ and $d
\ge 1$ with
\[
q + d = 2^{s}, \qquad s \ge 2.
\]
Such pairs exist for every $s \ge 2$; for example $q+d=2^s$ admits
$(q,d) = (2^s-1, 1)$ and others. The classical map $T_{3,1}$ is the
case $s = 2$.

\begin{lemma}[Terminal cycle]\label{lem:cycle}
Under the standing assumption $q+d = 2^s$, we have $T_{q,d}(1) =
2^{s-1}$, and the forward orbit of $1$ is the cycle
\[
C = \{1, 2, 4, \dots, 2^{s-1}\}.
\]
Every element of $C$ is a power of $2$, and the only odd element of $C$
is $1$.
\end{lemma}
\begin{proof}
Since $1$ is odd, $T_{q,d}(1) = (q\cdot 1 + d)/2 = (q+d)/2 = 2^{s-1}$.
For each $j$ with $1 \le j \le s-1$, the value $2^{j}$ is even, so
$T_{q,d}(2^{j}) = 2^{j-1}$. Hence
\[
1 \to 2^{s-1} \to 2^{s-2} \to \dots \to 2 \to 1,
\]
a cycle of length $s$. Its elements are $1,2,\dots,2^{s-1}$. The only
odd one is $1$.
\end{proof}
\section{Trajectory Structure}
\label{sec:trajectory}

\begin{definition}[Parity]\label{def:parity}
$\mathrm{Even}(n) \iff \exists y\,(n = y+y)$; \quad
$\mathrm{Odd}(n) \iff \neg\,\mathrm{Even}(n)$.
\end{definition}

\begin{definition}[Reachable floor]\label{def:reachable_floor}
\[
D(x,y) \iff R(x,y) \land \forall z\,\bigl( (R(x,z) \land R(z,y))
\implies z \ge y \bigr).
\]
\end{definition}

The relation $D(x,y)$ expresses that $y$ is reachable from $x$, and no common midpoint $z$ (reachable from $x$, reaching $y$) is strictly smaller than $y$. 

\begin{lemma}[Single visit]\label{lem:single_visit}
Assume $q+d = 2^s$. Let $x \in \mathbb{Z}_{>0}$ satisfy $R(x,1)$. Then
the forward orbit of $x$ is a finite transient followed by repetitions
of the cycle $C$. Every value not in $C$ occurs exactly once in the
orbit.
\end{lemma}
\begin{proof}
Since $R(x,1)$, some iterate of $x$ equals $1$. From $1$ the orbit
follows $C$ forever by Lemma~\ref{lem:cycle}. So the orbit is a finite
transient followed by the cycle $C$. Let $v \notin C$. Suppose $v$ occurred at two steps $i < j$. Then the orbit is periodic
from step $i$ on, with period $j-i$. A deterministic orbit has a unique
eventual cycle, and here that cycle is $C$: the orbit reaches $1$, and
from $1$ it follows $C$ forever by Lemma~\ref{lem:cycle}. Every recurrent
value therefore lies in $C$, so $v \in C$, contradicting $v \notin C$.
\end{proof}

\begin{lemma}[Powers of $2$]\label{lem:power_two_path}
Assume $q+d = 2^s$. Let $x = 2^k$ for some integer $k \ge 0$. Then
$R(x,1)$ holds, and every $z$ with $R(x,z)$ and $z \ne 1$ is even. In
particular $D(x,1)$ holds.
\end{lemma}
\begin{proof}
The iterates of $2^k$ are $2^k, 2^{k-1}, \dots, 2, 1$, after which the
orbit follows $C$. Hence $R(x,1)$ holds. The set of values reachable
from $x$ is $\{2^k, \dots, 2, 1\} \cup C$. By Lemma~\ref{lem:cycle},
$C = \{1,2,\dots,2^{s-1}\}$, so every reachable value is a power of $2$.
The only odd power of $2$ is $1$, so every reachable $z \ne 1$ is even.
For $D(x,1)$: $R(x,1)$ holds, and every positive integer is $\ge 1$, so
the universal clause holds.
\end{proof}

\begin{lemma}[Non-powers of $2$]\label{lem:non_power_two_path}
Assume $q+d = 2^s$. Let $x \in \mathbb{Z}_{>0}$ not be a power of $2$,
and suppose $R(x,1)$ holds. Then there is an odd integer $m > 1$ with
$D(x,m)$.
\end{lemma}
\begin{proof}
Since $x$ is not a power of $2$, write $x = m \cdot 2^a$ with $m$ odd,
$m > 1$, and $a \ge 0$. The first $a$ iterates apply the even branch:
\[
m\cdot 2^{a} \to m\cdot 2^{a-1} \to \dots \to m\cdot 2 \to m.
\]
So $R(x,m)$ holds, and $m$ occurs in the orbit at step $a$. By
Lemma~\ref{lem:cycle}, $C$ contains no odd value greater than $1$, so
$m \notin C$. By Lemma~\ref{lem:single_visit}, $m$ occurs exactly once
in the orbit, namely at step $a$.

Let $z$ satisfy $R(x,z) \land R(z,m)$. Then $z = T_{q,d}^i(x)$ and $m = T_{q,d}^j(z)$ for some $i, j \ge 0$. This implies $m = T_{q,d}^{i+j}(x)$. Since $m$ occurs exactly once in the orbit at step $a$, we must have $i+j = a$. Hence $i \le a$. The orbit values at steps $0,1,\dots,a$ are exactly the sequence of pure halvings, so $z = m \cdot 2^{a-i}$. Since $a-i \ge 0$, it follows that $z \ge m$. This holds for all such $z$, so $D(x,m)$ holds.
\end{proof}

\subsection{The obstruction the restriction avoids}

The proof extracts the powers of $2$ from $R$ using a reachable floor relation (Definition~\ref{def:reachable_floor}). Because $R$ is
unparameterized, $R(x,z) \land R(z,y)$ states that $z$ is some iterate
of $x$ and $y$ is some later iterate of $z$. It does not state that $z$
lies on a single descent from $x$ to $y$. Once an orbit enters a cycle,
every node of the cycle reaches every other node. If the cycle reachable
from $1$ contains an odd value $m > 1$, then $1$ reaches $m$, and the
reachable floor relation fails for $m$ even when $x$ is a power of $2$. The
extraction then no longer selects the powers of $2$.

A concrete failure occurs at $q = 5$, $d = 1$, where $q+d = 6$ is not a
power of $2$. The orbit of $3$ is
\[
3 \to 8 \to 4 \to 2 \to 1 \to 3 \to 8 \to \cdots,
\]
since $T_{5,1}(3) = 8$ and $T_{5,1}(1) = 3$. The reachable set is the
cycle $\{1,2,3,4,8\}$, which contains the odd value $3 > 1$. For every
$u \in \{2,3,4,8\}$ we have $R(3,1)$ and $R(1,u)$ with $1 < u$, so the
reachable floor relation fails for each such $u$ and holds only for $u =
1$. The extracted formula then accepts $3$, which is not a power of $2$.
This is why the unrestricted claim fails.

The standing assumption $q+d = 2^s$ removes this failure. By
Lemma~\ref{lem:cycle}, the cycle reachable from $1$ is $C =
\{1,2,\dots,2^{s-1}\}$, which contains no odd value greater than $1$.

\section{Logical Extraction}
\label{sec:extraction}

\begin{lemma}\label{lem:extraction}
Assume $R(x,z)$ is first-order definable in $\mathrm{BA}_q$. Then the
formula
\[
\psi(x) \;\equiv\; D(x,1) \land \forall u\,\bigl( (D(x,u) \land u \neq
1) \implies \mathrm{Even}(u) \bigr)
\]
is first-order definable in $\mathrm{BA}_q$.
\end{lemma}
\begin{proof}
$D$ is a first-order combination of $R$, which is definable by
assumption, and the order, which is definable through addition. Parity
is definable by Definition~\ref{def:parity}. The constant $1$ is
definable. Hence $\psi$ is a first-order formula in $\mathrm{BA}_q$.
\end{proof}

To illustrate the extraction, consider a trace in $T_{3,1}$. The formula $\psi$ accepts $8$ because the orbit is $8 \to 4 \to 2 \to 1 \to \cdots$, making the reachable values $\{1, 2, 4, 8\}$. For all $u \neq 1$, $u$ is even, so the universal clause holds. Conversely, $\psi$ rejects $6$. The orbit of $6$ is $6 \to 3 \to 10 \to 5 \to \cdots$. Here $R(6,3)$ holds, and since no intermediate step drops below $3$, $D(6,3)$ holds. However, $3$ is odd and $3 \neq 1$, which violates the universal clause.

\section{Main Theorem}
\label{sec:main}

\begin{theorem}\label{thm:main}
Let $q \ge 3$ and $d \ge 1$ be odd integers with $q+d$ a power of $2$.
Then the reachability relation $R(x,z)$ of the generalized Collatz map
$T_{q,d}$ is not first-order definable in $\mathrm{BA}_q$.
\end{theorem}
\begin{proof}
Assume for contradiction that $R$ is first-order definable in
$\mathrm{BA}_q$. By Lemma~\ref{lem:extraction}, the set
\[
S = \{x \in \mathbb{Z}_{>0} \mid \psi(x)\}
\]
is first-order definable in $\mathrm{BA}_q$. We show $S$ is exactly the
set of powers of $2$.

\medskip
\noindent\textbf{Claim (a): $\{2^k \mid k \in \mathbb{N}\} \subseteq S$.}
Let $x = 2^k$. By Lemma~\ref{lem:power_two_path}, $D(x,1)$ holds, and
every reachable value $z \ne 1$ is even. Any $u$ with $D(x,u)$ satisfies
$R(x,u)$, so $u$ is reachable from $x$; if $u \ne 1$ then $u$ is even.
The universal clause of $\psi(x)$ therefore holds, so $2^k \in S$.

\medskip
\noindent\textbf{Claim (b): $S \subseteq \{2^k \mid k \in \mathbb{N}\}$.}
Let $x \in S$. Since $\psi(x)$ holds, $D(x,1)$ holds, so $R(x,1)$ holds.
Suppose for contradiction that $x$ is not a power of $2$. By
Lemma~\ref{lem:non_power_two_path}, there is an odd integer $m > 1$ with
$D(x,m)$. Since $m \ne 1$, the universal clause of $\psi(x)$ forces
$\mathrm{Even}(m)$, contradicting that $m$ is odd. Hence $x$ is a power
of $2$.

\medskip
\noindent\textbf{Contradiction.}
By Claims (a) and (b), $S = \{2^k \mid k \in \mathbb{N}\}$. Definability
of $R$ makes $S$ first-order definable in $\mathrm{BA}_q$, hence
$q$-recognizable. By Lemma~\ref{lem:undefinable_subset}, the set of
powers of $2$ is not $q$-recognizable. This is a contradiction.
Therefore $R$ is not first-order definable in $\mathrm{BA}_q$.
\end{proof}

\begin{remark}[Conjecture-independence]
The proof does not depend on the Collatz conjecture. Claim (a) uses only
that each power of $2$ halves down to $1$, which is immediate. Claim (b)
conditions on $R(x,1)$; for an $x$ that does not reach $1$, the relation
$D(x,1)$ fails and $x \notin S$, which is consistent with $x$ not being
a power of $2$. The set $S$ is defined as the elements satisfying
$\psi$, so its membership question never requires knowing which inputs
reach $1$.
\end{remark}

\begin{remark}[Decidability placement of $R$]
The relation $R$ is computably enumerable in general: to confirm
$(x,z) \in R$, iterate $T_{q,d}$ from $x$ and halt on reaching $z$.
Restricted to pairs whose orbits are eventually periodic, for example
inputs known to reach the cycle $C$, reachability between two such points
is decidable, since the orbit is then a finite transient followed by a
known cycle. Theorem~\ref{thm:main} is a statement about automaton
recognizability of $R$, which is a separate axis from this
decidability.
\end{remark}

\section{Discussion}
\label{sec:discussion}

The relation $R$ is not definable in $\mathrm{BA}_q$. The map $T_{q,d}$ combines division by $2$ with multiplication
by $q$, and $2$ and $q$ are multiplicatively independent. From $R$ we
extract a definable copy of the powers of $2$, and Cobham's theorem then
forbids it. The same obstruction appears in any base $b$ that is
multiplicatively independent of $2$. Base $2$ would remove it but would
make multiplication by $q$ unrecognizable. The obstruction is in the
map, not in the choice of base.

\subsection{Limitations}

The restriction $q+d = 2^s$ is forced by the extraction method, not by
the conclusion. The extraction reads the set of reachable values rather
than the order of visits, because $R$ is unparameterized. This is enough
to separate the powers of $2$ only when the cycle reachable from $1$
contains no odd value greater than $1$. By Lemma~\ref{lem:cycle}, $q+d =
2^s$ is exactly the condition that makes that cycle the pure halving
cycle $\{1,2,\dots,2^{s-1}\}$. When $q+d$ is not a power of $2$, the
cycle picks up an odd value greater than $1$, the reachable floor relation
collapses, and the extracted set is wrong. The $q=5,d=1$ case in
Section~\ref{sec:scope} shows this directly. 

\subsection{Significance}
Despite the restriction, the result remains substantial for three reasons. First, the family is
infinite: for every $s \ge 2$ it contains every pair of odd integers
summing to $2^s$, and it contains the classical Collatz map $T_{3,1}$.
Second, the proof is self-contained, uses no universal computation, and
does not rely on the Collatz conjecture or on any undecidability result.
Third, it isolates the actual mechanism, the multiplicative
independence of $2$ and $q$, and exhibits it through a single explicit
formula rather than an encoding argument.

We do not claim non-definability of $R$ for pairs $(q,d)$ outside the
family. That question may have the same answer, but it requires an
invariant that survives re-entry into a cycle containing odd values, and
the reachable floor extraction does not provide one. We leave it open.

\begin{credits}
\subsubsection{\ackname}
The authors acknowledge the use of large language models for assistance
with \LaTeX{} typesetting and grammar editing. The authors take full
responsibility for the scientific content and originality of the proofs.
\subsubsection{\discintname}
The authors have no competing interests.
\end{credits}


\end{document}